\documentclass[11pt, a4paper]{article}
\usepackage{indentfirst}
\usepackage{amsfonts}
\usepackage{amssymb}
\usepackage{mathrsfs}
\usepackage{amsmath}
\usepackage{amsthm}
\usepackage{enumerate}
\usepackage{cite}
\usepackage{mathrsfs}
\allowdisplaybreaks
\usepackage{geometry}
\usepackage{ifpdf}
\ifpdf
\usepackage[colorlinks=true, linkcolor=blue, citecolor=red, final, backref=page, hyperindex]{hyperref}
\else
\usepackage[colorlinks, final, backref=page, hyperindex, hypertex]{hyperref}
\fi

\usepackage{txfonts}
\usepackage{indentfirst}
\usepackage{latexsym}
\usepackage[all]{xy}

\voffset=- 1.5cm
\def\<{\langle}
\def\>{\rangle}

\newtheorem{thm}{Theorem}[section]

\newtheorem{prop}[thm]{Proposition}
\newtheorem{ex}[thm]{Example}

\theoremstyle{definition}
\newtheorem{defn}{Definition}[section]

\theoremstyle{remark}
\newtheorem{re}{Remark}[section]
\begin{document}
	\title{\bf Cohomologies and deformations of weighted Rota-Baxter Lie algebras and associative algebras with derivations }
	\author{\bf  Basdouri Imed, Sadraoui Mohamed Amin, Shuangjian Guo}
	\author{{
			\  Basdouri Imed $^{1}$
			\footnote { Corresponding author, E-mail: basdourimed@yahoo. fr},
			\ Sadraoui Mohamed Amin $^{2}$
			\footnote { Corresponding author, E-mail: aminsadrawi@gmail.com}, 
			\  Shuangjian Guo $^{3}$
			\footnote { Corresponding author, E-mail: shuangjianguo@126.com}
		}\\
		\\
		{\small 1. University of Gafsa, Faculty of Sciences Gafsa, 2112 Gafsa, Tunisia} \\
		{\small 2. University of Sfax, Faculty of Sciences Sfax, BP
			1171, 3038 Sfax, Tunisia} \\
		{\small 3. School of Mathematics and statistics,Guizhou University of Finance and Economics, Guiyang, 550025} 
	}
	\date{}
	\maketitle
	\begin{abstract}
		The purpose of the present paper is to investigate cohomologies and deformations of weighted Rota-Baxter Lie algebras as well as weighted Rota-Baxter associative algebras with derivations. First we introduce a notion of weighted Rota-Baxter LieDer and weighted Rota-Baxter AssDer pairs. Then we construct cohomologies of weighted Rota-Baxter LieDer pairs, weighted Rota-Baxter AssDer pairs and we discuss the relation between their cohmologies. Finally, as an application, we study deformations of both of them.
	\end{abstract}

	\numberwithin{equation}{section}
	
	\tableofcontents
	\section{Introduction}
	 Rota-Baxter operators were first studied in the work of Baxter of the fluctuation theory in probability \cite{G0} and further was developed in \cite{G1}. These operators can be regarded as an algebraic abstraction representing the integral operator. Many papers have been devoted to various aspects of Rota-Baxter operators in many mathematical fields, like combinatorics \cite{N0}, renormalization in quantum field theory \cite{A0}, multiple zeta values in number theory \cite{L0}, Yang Baxter equations \cite{C0}, algebraic operad \cite{M0} and other papers. Rota-Baxter operators with arbitrary weight (also called weighted Rota-Baxter operators) was considered in \cite{C1,C2,J0}. \\
	 Deformation theory of some algebraic structure goes back to Gerstenhaber \cite{M1} for associative algebras
	 and Nijenhuis-Richardson \cite{A1} for Lie algebras. There are some advancements in deformation theory and cohomology theory of weighted Rota-Baxter algebras \cite{A2,K0}. More precisely, they considered weighted Rota-Baxter Lie algebras and associative algebras and define cohomology of them with coefficients in arbitrary Rota-Baxter representation.\\
	 Derivations have an important role to study many algebraic structures. Homotopy Lie algebras \cite{T0} and differential Galois theory \cite{A3} can be gained from derivations. Derivations also are important in control theory and gauge theories in quantum field theory \cite{V0,V1}. Recently, the cohomologies, extensions and deformations of Lie algebras with derivations (called LieDer pairs) were investigated in \cite{R0}. Then had been extended to associative algebras, Leibniz algebras, Lie triple systems, n-Lie algebras and compatible Lie algebras with derivations  in \cite{A4,A5,Q0,Q1,Q2,B0}. Also derivations play new role in twisting associative and nonassociative algebras to obtain InvDer algebraic structure, for more details see \cite{Bas} \\
	 Motivated by these works, we are generalizing the structre of Lie algebras with derivation \cite{R0} (respectively associative algebras with derivation \cite{A5}) and the structure of weighted Rota-Baxter Lie algebras \cite{A2} (respectively Rota-Baxter associative algebras \cite{K0}) to the cohomologies of weighted Rota-Baxter LieDer and AssDer pairs.\\
	 The paper is organized as follows. In section \ref{section 2}, we introduce a notion of a weighted Rota-Baxter LieDer pair and its representation. In section  \ref{section3}, we study cohomologies of weighted Rota-Baxter LieDer pairs and weighted Rota-baxter AssDer pairs. In section \ref{section 4}, we study deformation of two structures.
	 \section{Weighted Rota-Baxter LieDer pairs}\label{section 2}
	 Let $\mathfrak{L}=(L,[\cdot,\cdot])$ be a Lie algebra. A linear map $R:L\rightarrow L$ is said to be a $\lambda$-weighted Rota-Baxter operator if $R$ satisfies, for $\lambda \in \mathbb{K}$ 
	 	\begin{equation}\label{RBO eq1}
	 	[Rx,Ry]=R([Rx,y]+[x,Ry]+\lambda [x,y]),\quad \forall x,y \in L.
	 \end{equation}
	According to the previous operator, in this section we introduce the notion of $\lambda$-weighted Rota-Baxter LieDer pairs (or simply weighted Rota-Baxter LieDer pairs if there is no confusion) which is a LieDer pair $(\mathfrak{L},\delta)$ equipped with a $\lambda$-weighted Rota-Baxter operator. \\
This notion is a generalization of weighted Rota-Baxter Lie algebra $(\mathfrak{L},R)$, see \cite{L1,A2,A6,S0} for more details,  consisting of a Lie algebra $\mathfrak{L}$ together with a $\lambda$-weighted Rota-Baxter operator on it.
\subsection{Lie algebras equipped with a couple of  derivations}
In this subsection we study the structure of a Lie algebras $\mathfrak{L}$ equipped with a couple of  derivations $\delta_1,\delta_2$ and we investigate its relation with the LieDer pairs. Recall first the definition of LieDer pair.
\begin{defn}\cite{R0}
	A LieDer pair $(\mathfrak{L},\delta)$ is a Lie algebra $\mathfrak{L}$ equipped with a derivation $\delta$.
\end{defn}
So a Lie algebra equipped with a derivation leads to construct a LieDer pair. Next we combine a Lie algebra $\mathfrak{L}$ with two derivations $\delta_1$ and $\delta_2$ to construct a LieDer pair structure with an additionally condition.
\begin{re}
	Let $\mathfrak{L}$ be a Lie algebra and $\delta_1,\delta_2:L\rightarrow L$ be two derivations on $\mathfrak{L}$. Then $\delta_1\circ\delta_2$ is a derivation on $\mathfrak{L}$ if and only if for all $x,y\in L$ the following condition holds
	\begin{equation}\label{condition biderivation}
		[\delta_1x,\delta_2y]=[\delta_1y,\delta_2x]
	\end{equation} 
\end{re}
\begin{defn}
	Let $\mathfrak{L}$ be a Lie algebra, $\delta_1$ and $\delta_2$ be two derivations on it. Then $(\mathfrak{L},\delta_1,\delta_2)$ is called a \textbf{Lie-BiDer pair} if $(\mathfrak{L},\delta_1\circ\delta_2)$ is a LieDer pair.
\end{defn}
The previous definition means that $(\mathfrak{L},\delta_1,\delta_2)$ is a Lie-BiDer pair if and only if the condition \eqref{condition biderivation} is satisfied.
\begin{ex}
Let $L$ be two dimentional Lie algebra such that $[e_1,e_2]=e_2$. Then $(\mathfrak{L},\delta_1,\delta_2)$ is a Lie-BiDer pair where, for all $a,b,c,d\in \mathbb{K}$ 
		\begin{equation*}
			\delta_1=\begin{pmatrix}
				0 & 0 \\
				a & b 
			\end{pmatrix}\quad 
		\text{ and }\quad
		\delta_2=\begin{pmatrix}
			0 & 0 \\
			c & d 
		\end{pmatrix}	\quad 
	\text{ with }\quad ad=bc
		\end{equation*}
\end{ex}
\begin{defn}
	Let $(\mathfrak{L},\delta_1,\delta_2)$ and $(\mathfrak{K},\partial_1,\partial_2)$ be two Lie-BiDer pairs. A Lie-BiDer homomorphism $f:\mathfrak{L} \rightarrow \mathfrak{K}$ is a Lie algebra homomorphism from $\mathfrak{L}$ to $\mathfrak{K}$ such that 
	\begin{equation}\label{morphism Lie-BiDer pair}
		f\circ\delta_1=\partial_1\circ f,\quad f\circ\delta_2=\partial_2\circ f
	\end{equation} 
\end{defn}
\begin{defn}
	Let $(\mathfrak{L},\delta_1,\delta_2)$ be a Lie-BiDer pair. A representation of it on a vector space $V$ with respect to $(\varphi_1,\varphi_2)\in  \mathrm{gl}(V)$ is a Lie algebra morphism $\rho:L\rightarrow \mathrm{gl}(V)$ such that, for $x\in L$, we have 
	\begin{equation}\label{representation Lie-BiDer pair}
		\rho(\delta_1x)\circ\varphi_2=-\rho(\delta_2x)\circ\varphi_1.
	\end{equation}
Such a representation is denoted by $(\mathcal{V}=(V,\rho),\varphi_1,\varphi_2)$.
\end{defn}
\begin{ex}
	Let $\mathfrak{L}$ be a Lie algebra and a map $\mathrm{ad}_x:L \rightarrow L$ defined by 
	\begin{equation*}
		\mathrm{ad}_x(y)=[x,y],\quad \forall y\in L.
	\end{equation*}
	Then $(L,\mathrm{ad},\delta_1,\delta_2)$ is a representation of the Lie-BiDer pair $(L,\delta_1,\delta_2)$ on $L$ with respect to $\delta_1$ and $\delta_2$, it is called the adjoint representation.
\end{ex}
\begin{prop}\label{semidirect BiDer} Let $(\mathfrak{L},\delta_1,\delta_2)$ be a Lie-BiDer pair and $(\mathcal{V},\varphi_1,\varphi_2)$ be a representation of it. Then $(L\oplus V,\delta_1\oplus\varphi_1,\delta_2\oplus\varphi_2)$ is a Lie-BiDer pair with the following Lie structure 
	\begin{equation*}
		[x+a,y+b]_{\ltimes}=[x,y]+\rho(x)b-\rho(y)a;\quad x,y \in L \text{ and }a,b \in V.
	\end{equation*}
and the maps $\delta_\mathrm{i}\oplus\varphi_\mathrm{i}$	are given by 
\begin{equation*}
	\delta_\mathrm{i}\oplus\varphi_\mathrm{i}(x+a):=\delta_\mathrm{i}x+\varphi_\mathrm{i}(a),\quad \forall i\in \{1,2\}.
\end{equation*}

Such a Lie-BiDer pair is called the \textbf{semi-direct product} of $(\mathfrak{L},\delta_1,\delta_2)$ by a representation representation $(\mathcal{V},\varphi_1,\varphi_2)$ and it is denoted by $L\ltimes_{\textbf{BiDer}} V$.
\end{prop}
\begin{proof}
We need just to prove that the condition \eqref{condition biderivation} is satisfied for \textbf{semi-direct product}.	 
	 \begin{align*}
	 	[(\delta_1\oplus\varphi_1)(x+a),(\delta_2\oplus\varphi_2)(y+b)]_\ltimes&=[\delta_1x+\varphi_1(a),\delta_2y+\varphi_2(b)]_\ltimes\\
	 	&=[\delta_1x,\delta_2y]+\rho(\delta_1x)\varphi_2(b)-\rho(\delta_2y)\varphi_1(a)\\
	 	&\overset{\ref{condition biderivation}}{=}[\delta_1y,\delta_2x]+\rho(\delta_1x)\varphi_2(b)-\rho(\delta_2y)\varphi_1(a)\\
	 	&\overset{\ref{representation Lie-BiDer pair}}{=}[\delta_1y,\delta_2x]-\rho(\delta_2x)\varphi_1(b)+\rho(\delta_1y)\varphi_2(a)\\
	 	&=[(\delta_1\oplus\varphi_1)(y+b),(\delta_2\oplus\varphi_2)(x+a)]_\ltimes
	 \end{align*}
 This complete the proof.
\end{proof}

\subsection{Weighted Rota-Baxter LieDer pairs}
In this subsection, we introduce a notion of $\lambda$-weighted Rota-Baxter LieDer pair (or simply weighted Rota-Baxter LieDer pair) and some basic definitions. Denote the Lie algebra $(L,[\cdot,\cdot])$ by $\mathfrak{L}$ and its representation on a vector space $V$ by $\mathcal{V}=(V;\rho)$.\\

A weighted Rota-Baxter Lie algebra consists of a Lie algebra $\mathfrak{L}$ equipped with a Rota-Baxter operator of weight $\lambda$ denoted by $R$.\\
Inspired by the notion of LieDer pair \cite{R0} and the definition of weighted Rota-Baxter Lie algebra \cite{A6} we introduce the following. 
\begin{defn}
	A weighted Rota-Baxter LieDer pair consists of a LieDer pair $(\mathfrak{L},\delta)$ together with a $\lambda$-weighted Rota-Baxter operator $R$ such that
	\begin{equation}\label{condition1 RBLieDer pair}
		R\circ \delta=\delta\circ R.
	\end{equation} 
\end{defn}
\begin{ex}
	Let $\{e_1,e_2\}$ be a basis of a $2$-dimensional vector space $L$ over $\mathbb{K}$. Given a Lie structure $[e_1,e_2]=e_2$, then the triple $(\mathfrak{L},\delta,R)$ is a $\lambda$-weighted Rota-Baxter LieDer pair with 
	\begin{align*}
		\delta=\begin{pmatrix}
			0 & 0 \\
			0 & a 
		\end{pmatrix}\quad \text{ and   } \quad	R=\begin{pmatrix}
		0 & 0 \\
		0 & b
	\end{pmatrix},\quad \forall a,b \in \mathbb{K}.
	\end{align*}
\end{ex}
\begin{defn}
	A $\lambda$-weighted Rota-Baxter LieDer pair morphism from $(L_1,\delta_1,R_1)$ to $(L_2,\delta_2,R_2)$ is a Lie algebra morphism  $\varphi:L_1\rightarrow L_2$ such that the following identities holds, for all $x,y\in L_1$
	\begin{eqnarray}
		\varphi\circ \delta_1&=&\delta_2\circ\varphi,\label{morphism RBO2}\\
		\varphi\circ R_1&=&R_2\circ\varphi.\label{morphism RBO3}
	\end{eqnarray}
\end{defn}
Let $(\mathfrak{L},\delta)$ be a LieDer pair. Recall that a representation of it is a vector space $V$ with two linear maps $\rho:L\rightarrow gl(V)$ and $\delta_\mathrm{V}:L\rightarrow L$ such that, forall $x,y\in L$ 
\begin{eqnarray*}
	\rho([x,y])&=&\rho(x)\circ \rho(y)-\rho(y)\circ \rho(x),\\
	\delta_\mathrm{V}\circ \rho(x)&=&\rho(\delta x)+\rho(x)\circ \delta_\mathrm{V}.
\end{eqnarray*}

\begin{defn}
	Let $(\mathfrak{L},\delta,R)$ be a weighted Rota-Baxter LieDer pair. A representation of it is a triple $(\mathcal{V},\delta_\mathrm{V},T)$ where $T:V\rightarrow V$ is a linear map such that for all $x\in L$ and $u\in V$
	\begin{eqnarray}
		\rho([x,y])&=&\rho(x)\circ \rho(y)-\rho(y)\circ \rho(x),\label{Rep of RBDer pair0} \\
		\delta_V\circ \rho(x)&=&\rho(\delta x)+\rho(x)\circ \delta_V,\label{Rep of RBDer pair1}\\
		\rho(Rx)(Tu)&=&T(\rho(Rx)(u)+\rho(x)(Tu)+\lambda \rho(x)u),\label{Rep of RBDer pair2}\\
		T\circ\delta_V&=&\delta_V\circ T. \label{Rep of RBDer pair3}
	\end{eqnarray}
\end{defn}
\begin{ex}
	Let $(L,\delta,R)$ be a $\lambda$-weighted Rota-Baxter LieDer pair and $(\mathcal{V},\delta_\mathrm{V},T)$ be a representation of it. Then for any scalar $\mu\in \mathbb{K}$, the triple $(\mathcal{V},\delta_\mathrm{V},\mu T)$ is a representation of the $(\mu\lambda)$-weighted Rota-Baxter LieDer pair $(L,\delta_\mathrm{V},\mu R)$.
\end{ex}
\begin{ex}
	Let $(L,\delta,R)$ be a $\lambda$-weighted Rota-Baxter LieDer pair and $(\mathcal{V},\delta_\mathrm{V},T)$ be a representation of it. Then the quadruple $(\mathcal{V},\delta_\mathrm{V},-\lambda \mathrm{Id}_\mathrm{V}-T)$ is a representation of the $\lambda$-weighted Rota-Baxter LieDer pair $(L,\delta,-\lambda\mathrm{Id}_\mathrm{L}-R)$.
\end{ex}
\begin{ex}
	Any $\lambda$-weighted Rota-Baxter LieDer pair $(L,\delta,R)$ is a representation of itself. Such a representation is called the adjoint representation.
\end{ex}
Next, we construct the semi-direct product in the context of $\lambda$-weighted Rota-Baxter LieDer pair.
\begin{prop}
	Let $(\mathfrak{L},\delta,R)$ be a weighted Rota-Baxter LieDer pair and  $(\mathcal{V},\delta_\mathrm{V},T)$ be a representation of it. Then $(L\oplus V,\delta\oplus \delta_{\mathrm{V}},R\oplus T)$ is a weighted Rota-Baxter LieDer pair where the Lie bracket on $L\oplus V$ is given by 
	\begin{equation*}
		[x+a,y+b]_\ltimes:=[x,y]+\rho(x)b-\rho(y)a,
	\end{equation*}
and the derivation is given by 
	\begin{eqnarray*}
		\delta\oplus\delta_\mathrm{V}(x+a)=\delta x+\delta_\mathrm{V}a
	\end{eqnarray*}
and The $\lambda$-weighted Rota-Baxter LieDer pair is given by 
	\begin{eqnarray*}
	R\oplus T(x+a)=R x+Ta,\quad \forall x,y\in L \quad \forall a,b\in V.
\end{eqnarray*}
We call such structure by \textbf{the semi-direct product} of the $\lambda$-weighted Rota-Baxter LieDer pair $(L,\delta,R)$ by a representation of it  $(\mathcal{V},\delta_\mathrm{V},T)$ and denoted by $L\ltimes_{\mathrm{R.B.LieDer}}V$.
\end{prop}
\begin{proof}
	According to the \cite{A6} and \cite{R0} the proof is straightforward. The idea is to show that $\delta\oplus \delta_{\mathrm{V}}$ is a derivation on $L\oplus V$ and to show that $R\oplus T$ is a weighted Rota-Baxter operator.
\end{proof}
\begin{prop}
	Let $(\mathfrak{L},\delta,R)$ be a weighted Rota-Baxter LieDer pair. Then we have the triple $(L,[\cdot,\cdot]_\mathrm{R},\delta)$ is a LieDer pair with 
	\begin{equation}\label{induced RBLieDer}
		[x,y]_R:=[Rx,y]+[x,Ry]+\lambda [x,y]
	\end{equation}  
and it is denoted simply by $(\mathfrak{L}_\mathrm{R},\delta)$.
\end{prop}
\begin{proof}
	Let $x,y\in L$
	\begin{align*}
		\delta([x,y]_R)&=\delta([Rx,y]+[x,Ry]+\lambda [x,y])\\
		&=\delta([Rx,y])+\delta ([x,Ry]) +\lambda \delta([x,y])\\
		&=[\delta\circ Rx,y]+[Rx,\delta y]+[\delta x,Ry]+[x,\delta\circ R]+\lambda [\delta x,y]+\lambda [x,\delta y]\\
		&\overset{\eqref{condition1 RBLieDer pair}}{=}[R\circ \delta x,y]+[Rx,\delta y]+[\delta x,Ry]+[x,R\circ \delta]+\lambda [\delta x,y]+\lambda [x,\delta y]\\
		&=\Big([R\circ \delta x,y]+ [\delta x,Ry]+\lambda [\delta x,y]\Big)+\Big([Rx,\delta y]+[x,R\circ \delta]+\lambda [x,\delta y] \Big)\\
		&=[\delta x,y]_R+[x,\delta y]_R.
	\end{align*}	
This complete the proof.
\end{proof}
\begin{prop}
	The triple $(\mathfrak{L}_R,\delta,R)$ is a weighted Rota-Baxter LieDer pair and the map $R:L_R\rightarrow L$ is a morphism of weighted Rota-Baxter LieDer pair.
\end{prop}
\begin{proof}
	We have that $R$ is a $\lambda$-weighted Rota-Baxter operator on L, it follows then from  \eqref{RBO eq1} that 
	\begin{equation*}
		R([x,y]_R)=[Rx,Ry],\quad \forall x,y\in L.
	\end{equation*}
	This complete the proof.
\end{proof}

\begin{thm}\label{theorem needed in the cohomology}
	Let $(\mathfrak{L},\delta,R)$ be a weighted Rota-Baxter LieDer pair and $(\mathcal{V},\delta_{\mathrm{V}},T)$ be a representation of it. Define a map
	\begin{equation}\label{rep of new Rota-Baxter LieDer pairs}
		\widetilde{\rho}(x)(a)=\rho(Rx)(a)-T(\rho(x)(a)),\quad \text{for }x\in L , a\in V
	\end{equation} 
	Then $\widetilde{\rho}$ defines a representation of the LieDer pair $(L_{\mathrm{R}},\delta)$ on $(\widehat{\mathcal{V}},\delta_{\mathrm{V}})=((V;\widetilde{\rho}),\delta_{\mathrm{V}})$ if and only if the too conditions \eqref{condition1 RBLieDer pair} and \eqref{Rep of RBDer pair3} are satisfied.
Moreover, $(\widehat{\mathcal{V}},\delta_{\mathrm{V}},T)$ is a representation of the weighted Rota-Baxter LieDer pair $(L_{\mathrm{R}},\delta,R)$.
\end{thm}
\begin{proof}
	We have already, from \cite{A6}, that $\widetilde{\rho}$ is a representation of $L_\mathrm{R}$ on $V$ in the context of Lie algebra structure. So we need just to prove that equation \eqref{Rep of RBDer pair1} holds for the representation $\widetilde{\rho}$. Let $x\in L$ and $a\in V$
	\begin{align*}
		\delta_{\mathrm{V}}\circ \widetilde{\rho}(x)a&=\delta_{\mathrm{V}}\circ (\rho(Rx)a-T(\rho(x)a))\\
		&=\delta_{\mathrm{V}}\circ \rho(Rx)a-\delta_{\mathrm{V}}(T(\rho(x)a))\\
		&\overset{\ref{Rep of RBDer pair1}}{=}\rho(\delta\circ Rx)a+\rho(Rx)\circ \delta_{\mathrm{V}}a-\delta_{\mathrm{V}}(T(\rho(x)a))\\
		&\overset{\ref{Rep of RBDer pair3}}{=}\rho(\delta\circ Rx)a+\rho(Rx)\circ \delta_{\mathrm{V}}a-T(\delta_{\mathrm{V}}(\rho(x)a))\\
		&\overset{\ref{Rep of RBDer pair1}}{=}\rho(\delta\circ Rx)a+\rho(Rx)\circ \delta_{\mathrm{V}}a-T(\rho(\delta x)a+\rho(x)\circ \delta_{\mathrm{V}}a)\\
		&\overset{\ref{condition1 RBLieDer pair}}{=}\rho(R\circ\delta x)a-T(\rho(\delta x)a)+\rho(Rx)\circ \delta_{\mathrm{V}}a-T(\rho(x)\circ \delta_{\mathrm{V}}a)\\
		&=\widetilde{\rho}(\delta x)a+\widetilde{\rho}(x)\circ \delta_{\mathrm{V}}a.
	\end{align*}
This means that $(\widetilde{\mathcal{V}},\delta_{\mathrm{V}})$ is a representation of the LieDer pair $(L_\mathrm{R},\delta)$. For the next result see (Theorem 2.17 in \cite{A6} ).
\end{proof}

\section{Cohomology of weighted Rota-Baxter LieDer and AssDer pairs}\label{section3}
In this section we introduce the cohomology of weighted Rota-Baxter LieDer and AssDer pairs.
\subsection{Cohomology of weighted Rota-Baxter LieDer pairs}
In this subsection, we first recall the Chevally-Eilenberg cohomology of Lie algebras, the cohomology of  weighted Rota-Baxter Lie algebras \cite{A6} and the cohomology of LieDer pairs \cite{R0} with coefficients in an arbitrary representation. Then we define the cohomology of $\lambda$-weighted Rota-Baxter LieDer pairs.\\
Let $\mathfrak{L}=(L,[\cdot,\cdot])$ be a Lie algebra, the Chevally-Eilenberg cohomology of $\mathfrak{L}$ with cofficents in the representation $\mathcal{V}$ is given by the cohomology of the cochain complex $\{\mathrm{C}^{\star}(L;\mathcal{V}),\mathrm{d}\}$ where $\mathrm{C}^{n}(L;\mathcal{V})=\mathrm{Hom}(\wedge^nL,V)$ for $n\geq0$ and the coboundary map 
\begin{equation*}
	\mathrm
	{d}:C^n(L;\mathcal{V})\rightarrow C^{n+1}(L;\mathcal{V})
\end{equation*}
is given by 
\begin{align*}
	(\mathrm
	{d}(f_n))(x_1,\ldots,x_{n+1})&=\displaystyle\sum_{i=1}^{n+1}(-1)^{i+n}\rho(x_i)f_n(x_1,\ldots,\hat{x_i},\ldots,x_{n+1})\\
	&+\displaystyle\sum_{1\leq i<j\leq n+1}(-1)^{i+j+n+1}f_n([x_i,x_j],x_1,\ldots,\hat{x_i},\ldots,\hat{x_j},\ldots,x_{n+1})
\end{align*}
for $f_n\in C^n({L;\mathcal{V}})$ and $x_1,\ldots,x_{n+1}\in L$.\\
Let $(\mathfrak{L},\delta)$ be a LieDer pair. Recall that the cohomology of LieDer pair $(\mathfrak{L},\delta)$ with coefficents in a representation $(\mathcal{V},\delta_{\mathrm{V}})$ is given as follow: \\
The set of LieDer pair $0$-cochains is $0$ and the set of LieDer pair $1$-cochains is $\mathfrak{C}^1_{\mathrm{LieDer}}(L;\mathcal{V})=\mathrm{Hom}(L,V)$. For $n\geq2$, the set of LieDer pair $n$-cochains is given by 
\begin{equation*}
	\mathfrak{C}^n_{\mathrm{LieDer}}(L;\mathcal{V}):=C^n(L;\mathcal{V})\times C^{n-1}(L;\mathcal{V})
\end{equation*}
For $n\geq1$, define the following operator $\partial:C^n(L;\mathcal{V})\rightarrow C^n(L;\mathcal{V})$ by 
\begin{equation*}
	\partial f_n=\displaystyle\sum_{i=1}^nf_n\circ (\mathbf{1}\otimes \cdots\otimes \underbrace{\delta}_{\text{i-th place}}\otimes \cdots\otimes \mathbf{1})-\delta_{\mathrm{V}}\circ f_n.
\end{equation*}
Define $\mathrm{D}:\mathfrak{C}^1_{\mathrm{LieDer}}(L;\mathcal{V})\rightarrow \mathfrak{C}^2_{\mathrm{LieDer}}(L;\mathcal{V})$ by 
\begin{equation*}
	\mathrm{D}f_1=(\mathrm
	{d}(f_1),(-1)^1\partial f_1),\quad \forall f_1\in \mathrm{Hom}(L,V).
\end{equation*}
And for $n\geq2$, define $\mathrm{D}:\mathfrak{C}^n_{\mathrm{LieDer}}(L;\mathcal{V})\rightarrow \mathfrak{C}^{n+1}_{\mathrm{LieDer}}(L;\mathcal{V})$ by 
\begin{equation*}
	\mathrm{D}(f_n,g_{n-1})=(\mathrm{d}(f_n),\mathrm{d}(g_{n-1})+(-1)^n\partial f_n)
\end{equation*}
for all $f_n\in C^n(L;\mathcal{V})$ and $g_{n-1}\in C^{n-1}(L;\mathcal{V})$.\\
Recall also the following equation
\begin{equation}\label{coboundary1}
	\mathrm{d}\circ \partial =\partial \circ \mathrm{d}
\end{equation}

 Let $(\mathfrak{L},R)$ be a Rota-Baxter Lie algebra, the cohomology of weighted Rota-Baxter Lie algebra with coefficients in a representation $(\mathcal{V},T)$ is given as follow as follow :\\
For each $n\geq 0$, define an abelian group $\mathfrak{C}^n_{\mathrm{R}}(L;\mathcal{V})$ by 
		$$ \mathfrak{C}^n_{\mathrm{R}}(L;V)=
\left\{
\begin{array}{ll}
	C^0(L;\mathcal{V})=V,& \text{ if } n=0,\\
	C^n(L;\mathcal{V})\oplus C^{n-1}(L_{\mathrm{R}};\widetilde{\mathcal{V}})=\mathrm{Hom}(\wedge^nL,V)\otimes\mathrm{Hom}(\wedge^{n-1}L,V),& \text{ if } n\geq1.	
\end{array}
\right.
$$
Where $\mathfrak{L}_{\mathrm{R}}=(L,[\cdot,\cdot]_{\mathrm{R}},R)$ is a Rota-Baxter Lie algebra with the bracket is defined in \eqref{induced RBLieDer} and $\widetilde{\mathcal{V}}=(V,\widetilde{\rho})$ is a representation of it with $\widetilde{\rho}$ is defined in \eqref{rep of new Rota-Baxter LieDer pairs}.\\
The coboundary map is defined as $\mathfrak{d}_{\mathrm{R}}:\mathfrak{C}^n_{\mathrm{R}}(L;\mathcal{V})\rightarrow \mathfrak{C}^{n+1}_{\mathrm{R}}(L;\mathcal{V})$ by 
		$$ 
\left\{
\begin{array}{ll}
	\mathfrak{d}_{\mathrm{R}}(v)=(\mathrm{d}(v),-v),& \text{ for } v\in \mathfrak{C}^0_{\mathrm{R}}(L;\mathcal{V})=V,\\
	\mathfrak{d}_{\mathrm{R}}(f_n,g_{n-1})=(\mathrm{d}(f_n),-\mathrm{d}_{\mathrm{R}}(g_{n-1})-\Phi^n(f_n)),& \text{ for } (f_n,g_{n-1})\in\mathfrak{C}^n_{\mathrm{R}}(L;\mathcal{V}) .	
\end{array}
\right.
$$
With $\mathrm{d}_{\mathrm{R}}:C^n(L_{\mathrm{R}};\widetilde{\mathcal{V}})\rightarrow C^{n+1}(L_{\mathrm{R}};\widetilde{\mathcal{V}})$ is given by 
\begin{align*}
	(\mathrm{d}_{\mathrm{R}}f_n)(x_1,\ldots,x_{n+1})&=\displaystyle\sum_{i=1}^{n+1}(-1)^{i+n} \ \widetilde{\rho}(x_i)f_n(x_1,\ldots,\hat{x_i},\ldots,x_{n+1})\\
	&+\displaystyle\sum_{1\leq i<j\leq n+1}(-1)^{i+j+n+1}f_n([x_i,x_j]_{\mathrm{R}},x_1,\ldots,\hat{x_i},\ldots,\hat{x_j},\ldots,x_{n+1})
\end{align*}
and $\Phi:C^n(L;\mathcal{V})\rightarrow C^n(L_{\mathrm{R}};\widetilde{\mathcal{V}})$, is inspired from \cite{K0}, defined by

		\begin{equation}\label{cohomology of RBO}			
\left\{
\begin{array}{ll}
	\Phi=\mathbf{1}_{\mathrm{V}},& \\
	\Phi(f_n)(x_1,\cdots,x_n)=f_n(Rx_1,\cdots,Rx_n)-\displaystyle\sum_{k=0}^{n-1}\lambda^{n-k-1}\displaystyle\sum_{i_1<\ldots<i_k}T\circ f_n(x_1,\cdots,R(x_{i_1}),\cdots,R(x_{i_k}),\cdots,x_n).&  	
\end{array}
\right.
\end{equation}

is a morphism of cochain complex from $\{C^{\star}(L;\mathcal{V}),\mathrm{d}\}$ to $\{C^{\star}(L_{\mathrm{R}};\widetilde{\mathcal{V}}),\mathrm{d}\}$ i.e. 
\begin{equation}\label{coboundary2}
	\mathrm{d}_{\mathrm{R}}\circ\Phi=\Phi\circ \mathrm{d},\quad \forall n\geq0.
\end{equation}
Also since $L_{\mathrm{R}}$ is a Lie algebra and $\mathrm{d}_{\mathrm{R}}$ its coboundary with respect to the representation $\widetilde{\mathcal{V}}=(V;\widetilde{\rho})$ and  $(L_{\mathrm{R}},\delta)$ is a LieDer pair then, from the cohomology of LieDer pair (\cite{R0}. Lemma(3.1)), we get
\begin{equation}\label{coboundary3}
\partial\circ\mathrm{d}_{\mathrm{R}}=\mathrm{d}_{\mathrm{R}}\circ\partial
\end{equation}
Using all those tools we are in position to define the cohomology of $\lambda$-weighted Rota-Baxter LieDer pair $(\mathfrak{L},\delta,R)$ with coefficients in a representation $(\mathcal{V},\delta_\mathrm{V},T)$. The set of $\lambda$-weighted Rota-Baxter LieDer pair $0$-cochains is $0$ and the set of $\lambda$-weighted Rota-Baxter LieDer pair $1$-cochains to be $\mathfrak{C}^1_{\mathrm{RBLieDer}}(L,\mathcal{V})=\mathrm{Hom}(L,V)$.\\ For $n\geq2$ define $\lambda$-weighted Rota-Baxter LieDer pair $n$-cochains by 
\begin{equation*}
	\mathfrak{C}^n_{\mathrm{RBLieDer}}(L;\mathcal{V}):=\mathfrak{C}^n_{\mathrm{LieDer}}(L;\mathcal{V})\otimes C^{n-1}(L_R;\widetilde{\mathcal{V}})
\end{equation*}
Define 
\begin{equation*}
	\mathfrak{D}_{\mathrm{RBLieDer}}:\mathfrak{C}^n_{\mathrm{RBLieDer}}(L,\mathcal{V})\rightarrow \mathfrak{C}^{n+1}_{\mathrm{RBLieDer}}(L,\mathcal{V})
\end{equation*}
as follow
\begin{enumerate}
	\item[$\bullet$] For $n=1$
	\begin{eqnarray*}
		\mathfrak{D}_{\mathrm{RBLieDer}}&:&\mathfrak{C}^1_{\mathrm{RBLieDer}}(L,\mathcal{V})\rightarrow \mathfrak{C}^2_{\mathrm{RBLieDer}}(L,\mathcal{V}) \text{  is given by}\\
		&&\mathfrak{D}_{\mathrm{RBLieDer}}(f)=(d(f),-\partial (f),-\Phi (f)),\quad \forall f\in \mathrm{Hom}(L,V).
	\end{eqnarray*}
    \item [$\bullet$] For $n\geq2$ 
    \begin{eqnarray*}
    	\mathfrak{D}_{\mathrm{RBLieDer}}&:&\mathfrak{C}^n_{\mathrm{RBLieDer}}(L,\mathcal{V})\rightarrow \mathfrak{C}^{n+1}_{\mathrm{RBLieDer}}(L,\mathcal{V}) \text{  is given by}\\
    	&&\mathfrak{D}_{\mathrm{RBLieDer}}((f,g),h)=(d(f),d(g)+(-1)^n\partial (f),-d_R(h)-\Phi f),\quad \forall ((f,g),h)\in \mathrm{Hom}(L,V).
    \end{eqnarray*}
\end{enumerate}
Next in the following theorem we are in position to prove that $\mathfrak{D}_{\mathrm{RBLieDer}}$ is a coboundary map.
\begin{thm}
	The map $\mathfrak{D}_{\mathrm{RBLieDer}}$ is a coboundary operator, means that 
	\begin{equation*}
		\mathfrak{D}_{\mathrm{RBLieDer}}\circ \mathfrak{D}_{\mathrm{RBLieDer}}=0.
	\end{equation*}
\end{thm}
\begin{proof}
	For $n\geq1$, using equations \eqref{coboundary1} and \eqref{coboundary2}
	\begin{align*}
		\mathfrak{D}_{\mathrm{RBLieDer}}\circ\mathfrak{D}_{\mathrm{RBLieDer}}((f,g),h)&=\mathfrak{D}_{\mathrm{RBLieDer}}(d(f),d(g)+(-1)^n\partial (f),-d_R(h)-\Phi f)\\
		&=(d^2(f),d^2(g)+(-1)^nd\circ\partial(f)+(-1)^{n+1}(f),d^2_R(h)+d_R\circ\Phi(f)-\Phi\circ d(f))\\
		&=(0,0+(-1)^nd\circ\partial(f)+(-1)^{n+1}(f),0+d_R\circ\Phi(f)-\Phi\circ d(f))\\
		&=(0,0,0)
	\end{align*}
This complete the proof.
\end{proof}
With respect to the representation $(\mathcal{V},\delta_V,T)$ we obtain a complex $\{\mathfrak{C}^{\star}_{\mathrm{RBLieDer}}(L,\mathcal{V}),\mathfrak{D}_{\mathrm{RBLieDer}}\}$. Let $\mathcal{Z}^n_{\mathrm{RBLieDer}}(L,\mathcal{V})$ and $\mathcal{B}^n_{\mathrm{RBLieDer}}(L,\mathcal{V})$ denote the space of $n$-cocycles and $n$-coboundaries, respectively. Then we define the corresponding cohomology groups by 
\begin{equation*}
	\mathcal{H}^n_{\mathrm{RBLieDer}}(L,\mathcal{V}):=\frac{\mathcal{Z}^n_{\mathrm{RBLieDer}}(L,\mathcal{V})}{\mathcal{B}^n_{\mathrm{RBLieDer}}(L,\mathcal{V})},\quad \text{for } n\geq0.
\end{equation*}
They are called the cohomolgy of $\lambda$-weighted Rota-Baxter LieDer pair $(\mathfrak{L},\delta,R)$ with coefficients in the representation $(\mathcal{V},\delta_{\mathrm{V}},T)$.\\
\begin{re}
	Recall that from \cite{A6}, if $(\mathfrak{L},R)$ is a Rota-Baxter Lie algebra and $(\mathcal{V},T)$ is a representation of it. An element $v\in V$ is in $\mathrm{Z}^0_{RB}(L,\mathcal{V})$ if and only if $(d_{\mathrm{R}}(v),-v)=0$, this holds when $v=0$ which means that $\mathrm{H}^0_{RB}(L,\mathcal{V})=0$ and it coincide in our paper with $\mathrm{H}^0_{RBLieDer}(L,\mathcal{V})=0=\mathrm{H}^0_{RB}(L,\mathcal{V})$. \\
	Also $\mathrm{H}^1_{RBLieDer}(L,\mathcal{V})$ coincide with $\mathrm{H}^1_{RB}(L,\mathcal{V})=\frac{\mathrm{Der}(L,\mathcal{V})}{\mathrm{InnDer}(L,\mathcal{V})}$.
\end{re}
\begin{prop}
	Let $(\mathcal{V},\delta_{\mathrm{V}},T)$ be a representation of a $\lambda$-weighted Rota-Baxter LieDer pair $(\mathfrak{L},\delta,R)$. Then we have 
	\begin{equation*}
		\mathcal{H}^1_{\mathrm{RBLieDer}}(L,\mathcal{V})=\{f;\quad f\in\mathcal{Z}^1(L,\mathcal{V}),\quad f\circ\delta=\delta_V\circ f,\quad f\circ R=T\circ f\}.
	\end{equation*}
\end{prop}
\subsection{Cohomology of weighted Rota-Baxter AssDer pairs}
In \cite{K0} authors defined the cohomology of a weighted Rota-Baxter associative algebra with coefficients in a Rota-Baxter bimodule. Later Das.A studied the associative algebra with derivation and denote it by AssDer pairs \cite{A5}.\\
In this subsection we study the cohomology of $\lambda$-weighted Rota-Baxter AssDer pair and we show that this cohomolgy is related to the cohomology of Rota-Baxter LieDer pairs by suitable skew-symmetrization.\\
Let $(\mathfrak{A}=(A,\mu))$ be an associative algebra and $\delta:A\rightarrow A$ be a derivation on it. Recall that an AssDer  pair $(\mathfrak{A},\delta)$ is an associative algebra equipped with the derivation $\delta$.\\
 A linear map $R:A \rightarrow A$ is said to be a $\lambda$-weighted Rota-Baxter operator on $\mathfrak{A}$ if it satisfies 
\begin{equation}\label{RBO on Ass Algebras}
	\mu(Rx,Ry)=R(\mu(Rx,y)+\mu(x,Ry)+\lambda \mu(x,y)),\quad \forall x,y\in L.
\end{equation}
A bimodule over $\mathfrak{A}$ consists of a vector space $M$ together with two linear maps $l:A\otimes M\rightarrow M$ and $r:M\otimes A\rightarrow M$ such that for all $x,y\in A$ and $m\in M$
\begin{eqnarray*}
	l(\mu(x,y),m)&=&l(x,l(y,m)),\\
	r(l(x,m),y)&=&l(x,r(m,y))\\
	r(r(m,x),y)&=&r(m,\mu(x,y)).
\end{eqnarray*}
We will write $xm$ instead of $l(x,m)$ and $mx$ instead of $r(m,x)$ when there are no confusions.
\begin{defn}
	A weighted Rota-Baxter AssDer pair consisits of an AssDer pair $(\mathfrak{A},\delta)$ together with a weighted Rota-Baxter operator such that 
	\begin{equation*}
		R\circ \delta=\delta\circ R.
	\end{equation*}
\end{defn}
\begin{defn}
	Let $(\mathfrak{A}_1,\delta_1,R_1)$ and $(\mathfrak{A}_2,\delta_2,R_2)$ be two $\lambda$-weighted Rota-Baxter AssDer pairs. Then the linear map $f:A_1\rightarrow A_2$ is said to be a $\lambda$-weighted Rota-Baxter homomorphism if is satisfies the following, for $x,y\in A$
	\begin{eqnarray*}
		f\circ\mu_1(x,y)&=&\mu_2(f(x),f(y)),\\
		f\circ\delta_1&=&\delta_1\circ f,\\
		f\circ R_1&=&R_1\circ f.
	\end{eqnarray*}
\end{defn}
\begin{defn}
	Let $(\mathfrak{A},\delta,R)$ be a $\lambda$-weighted Rota-Baxter AssDer pair. A bimodule (representation) over it is a triple $(M,\delta_\mathrm{M},T)$ which is both a left and right module on $(\mathfrak{A},\delta,R)$ and $M$ is an $A$-bimodule. This means the followings, for all $x,y\in A$ and $m\in M$
	\begin{eqnarray}
		\delta_\mathrm{M}(xm)&=&\delta(x)m+x\delta_\mathrm{M}(m),\label{AssDer rep 1}\\
		\delta_\mathrm{M}(mx)&=&\delta_\mathrm{M}(m)x+m\delta(x)\label{AssDer rep 2}\\
		R(x)T(m)&=&T(R(x)m+xT(m)+\lambda xm)\\
		T(m)R(x)&=&T(T(m)x+mR(x)++\lambda mx)\\
		\delta_\mathrm{M}\circ T&=&T\circ \delta_\mathrm{M} \label{AssDer rep 5}.
	\end{eqnarray} 
\end{defn}
\begin{prop}
	Let $(\mathfrak{A},\delta,R)$ be a $\lambda$-weighted Rota-Baxter AssDer pair and $(M,\delta_\mathrm{M},T)$ be a representation of it. Then $(A\oplus M,\delta\oplus\delta_\mathrm{M},R\oplus T)$ is a $\lambda$-weighted Rota-Baxter AssDer pair where its structure is given by for all $x,y\in A$ and $m,n\in M$
	\begin{eqnarray*}
		\mu_\ltimes(x+m,y+n)&=&\mu(x,y)+xn+my,\\
		(\delta\oplus\delta_\mathrm{M})(x+m)&=&\delta(x)+\delta_\mathrm{M}(m)\\
		(R\oplus T)(x+m)&=&R(x)+T(m).
	\end{eqnarray*}
\end{prop}
\begin{proof}
	the proof is easy to check because it is known that $A\oplus M$ with the product $\mu_\ltimes$ and the linear map $\delta\oplus\delta_\mathrm{M}$ is a LieDer pair, see \cite{A5}. And it is well know that $R\oplus T$ is a $\lambda$-weighted Rota-Baxter operator.
\end{proof}
\begin{prop}
	Let $(\mathfrak{A},\delta,R)$ be a weighted Rota-Baxter AssDer pair. Define a new binary operation as follow :
	\begin{equation}
		\mu_\mathrm{R}:=\mu(x,Ry)+\mu(Rx,y)+\lambda \mu(x,y),\quad \forall x,y\in A.
	\end{equation}
Then
\begin{enumerate}
	\item[1)] The operation $\mu_\mathrm{R}$ forms an AssDer pair $(A,\mu_\mathrm{R},\delta)$ together with the derivation $\delta$.
	\item [2)] The quadruple $(A,\mu_\mathrm{R},\delta,R)$ forms a weighted Rota-Baxter AssDer pair and it is denoted $(\mathfrak{A}_\mathrm{R},\delta,R)$.
	\item[3)] The map $R:(A,\mu_\mathrm{R},\delta,R)\rightarrow(A,\mu,\delta,R)$ is a morphism of weighted Rota-Baxter AssDer pairs.
\end{enumerate}
\end{prop}
\begin{proof}
	All we need to prove is that $\delta$ is a derivation on the operation $\mu_\mathrm{R}$, the rest of the proof see (\cite{L1},Theorem 1.1.17).\\
	Let $x,y\in A$
	\begin{align*}
		\delta(\mu_\mathrm{R}(x,y))&=\delta(\mu(x,Ry)+\mu(Rx,y)+\lambda \mu(x,y))\\
		&=\mu(\delta x,Ry)+\mu(x,\delta\circ Ry)+\mu(\delta\circ Rx,y)+\mu(Rx,\delta y)+\lambda \mu(\delta x,y)+\lambda \mu(x,\delta y)\\
		&\overset{\ref{condition1 RBLieDer pair}}{=}\mu(\delta x,Ry)+\mu(R\circ \delta x,y)+\lambda \mu(\delta x,y)+\mu(x,R\circ\delta y)+\mu(Rx,\delta y)+\lambda \mu(x,\delta y)\\
		&=\mu_\mathrm{R}(\delta x,y)+\mu_\mathrm{R}(x,\delta y).
	\end{align*}
This complete the proof.
\end{proof}

\begin{prop}
	Let $(\mathfrak{A},\delta,R)$ be a weighted Rota-Baxter AssDer pair and $(M,\delta_\mathrm{M},T)$ be weighted Rota-Baxter bimodule over it. Define a left action $l_\mathrm{R}$ and a right action $r_\mathrm{R}$ of $A$ on $M$ as follows, for all $x\in A,m\in M$
	\begin{eqnarray*}
		l_\mathrm{R}(x,m)&=&l(Rx,m)-T(L(x,m))\\
		r_\mathrm{R}(m,x)&=&r(m,Rx)-T(r(m,x)).
	\end{eqnarray*}
Then these actions makes $M$ into weighted Rota-Baxter bimodule over $(\mathfrak{A}_\mathrm{R},\delta,R)$ and denote it this new bimodule by $(M_\mathrm{R},\delta_M,T)$.
\end{prop}
\begin{proof}
	It is known that $l_\mathrm{R}$ (respectively $r_\mathrm{R}$) is a left action (respectively right action) of $A$ on $M_\mathrm{R}$. So we need just to prove equations \eqref{AssDer rep 1} and \eqref{AssDer rep 2}.\\
	Let $x\in A,m\in M$
	\begin{align*}
		\delta_\mathrm{M}\circ l_\mathrm{R}&=\delta_\mathrm{M}(l(R(x),m)-T(l(x,m)))\\
		&\overset{\ref{condition1 RBLieDer pair}}{=}l(R\circ \delta x,m)+l(Rx,\delta_\mathrm{M}(m))-T\circ \delta_\mathrm{M}(l(x,m))\\
		&\overset{\ref{AssDer rep 5}}{=}l(R\circ \delta x,m)-T(l(\delta x,m))+l(Rx,\delta_\mathrm{M}(m))-T(l(x,\delta_\mathrm{M}(m)))\\
		&=l_\mathrm{R}(\delta x,m)+l_\mathrm{R}(x,\delta_\mathrm{M}(m))
	\end{align*}Similarly to $r_\mathrm{R}$.
\end{proof}
Next we will define a cohomology for weighted Rota-Baxter AssDer pairs.\\
Let $(\mathfrak{A},\delta)$ be an AssDer pair and $(M,\delta_M)$ be a representation of it. Recall from \cite{K0}, the Hochshild cochain complex of $\mathfrak{A}_\mathrm{R}$ with coefficients in $M_\mathrm{R}$ is given as follow 
\begin{equation*}
	\mathrm{d_{R.Hoch}}:C^n(\mathfrak{A}_\mathrm{R},M_\mathrm{R})\rightarrow C^{n+1}(\mathfrak{A}_\mathrm{R},M_\mathrm{R}) \text{ by }
\end{equation*}
\begin{align*}
		\mathrm{d_{R.Hoch}}(f)(x_1,\ldots,x_{n+1})&=
	l_\mathrm{R}(x_1,f(x_2,\ldots,x_{n+1}))+\displaystyle\sum_{i=1}^n(-1)^if(x_1,\ldots,x_{i-1},\mu_\mathrm{R}(x_i,x_{i+1}),\ldots,x_{n+1})\\
	&+(-1)^{n+1}r_\mathrm{R}(f(x_1,\ldots,x_n),x_{n+1}).
\end{align*}
Define $\partial$ the coboundary operator from $C^n(A,M)$ to $C^n(A,M)$ and it is given by 
\begin{equation*}
	\partial (f)=\displaystyle\sum_{i=1}^n f\circ (\mathbf{1}\otimes \ldots \delta \otimes \ldots \otimes \mathbf{1})-\delta_\mathrm{M}\circ f.
\end{equation*}
First step let's define the cohomology, in context of AssDer pair structure, of the weighted Rota-Baxter operator $R$ with coefficients in the representation $(M_\mathrm{R},\delta_\mathrm{M},T)$.\\
Consider the cochain complex of the AssDer pair $(\mathfrak{A}_\mathrm{R},\delta)$ with coefficients in the representation $(M_\mathrm{R},\delta_M)$ 
\begin{equation*}
	\mathfrak{C}_\mathrm{AssDer}^\star(\mathfrak{A}_\mathrm{R},M_\mathrm{R}):=\oplus_{n\geq0}\mathfrak{C}_\mathrm{AssDer}^n(\mathfrak{A}_\mathrm{R},M_\mathrm{R})
\end{equation*}
The space $\mathfrak{C}_\mathrm{AssDer}^n(\mathfrak{A}_\mathrm{R},M_\mathrm{R})$ of $n$-cochains is defined as follow 
		$$
\left\{
\begin{array}{ll}
	\mathfrak{C}_\mathrm{AssDer}^0(\mathfrak{A}_\mathrm{R},M_\mathrm{R})=0,&\\
	\mathfrak{C}_\mathrm{AssDer}^1(\mathfrak{A}_\mathrm{R},M_\mathrm{R})=\mathrm{Hom}(\mathfrak{A}_\mathrm{R},M_\mathrm{R}),&\\
		\mathfrak{C}_\mathrm{AssDer}^n(\mathfrak{A}_\mathrm{R},M_\mathrm{R})=C^n(\mathfrak{A}_\mathrm{R},M_\mathrm{R})\oplus C^{n-1}(\mathfrak{A}_\mathrm{R},M_\mathrm{R}),& \text{for } n\geq2.
\end{array}
\right.
$$
With the coboundary map is given by $\mathfrak{d}_{R.A.D}:\mathfrak{C}_\mathrm{R.AssDer}^n(\mathfrak{A}_\mathrm{R},M_\mathrm{R})\rightarrow \mathfrak{C}_\mathrm{R.AssDer}^{n+1}(\mathfrak{A}_\mathrm{R},M_\mathrm{R})$ as follow

\begin{equation}\label{coboundary AssDer of d.R}
	\mathfrak{d}_{R.A.D}(f,g)=(d_\mathrm{R.Hoch}(f),d_\mathrm{R.Hoch}(g)+(-1)^n\partial (f)),\quad \forall (f,g)\in \mathfrak{C}_\mathrm{R.AssDer}^n(\mathfrak{A}_\mathrm{R},M_\mathrm{R}).
\end{equation}

By \cite{R0} and since  $d_\mathrm{R.Hoch}$ is Hochschild coboundary \cite{K0}, it is easy to check that 
\begin{eqnarray}
	\mathfrak{d}_{R.A.D}\circ \mathfrak{d}_{R.A.D}&=&0,\\
	d_\mathrm{R.Hoch}\circ \partial&=&\partial \circ d_\mathrm{R.Hoch}.
\end{eqnarray}
with the previous result we are able to define the cohomology of Rota-Baxter operator $R$ with coefficients in $(M_\mathrm{R},\delta_\mathrm{M},T)$ in the context of AssDer pair structures.
\begin{defn}
	Let $(\mathfrak{A},\delta,R)$ be a weighted Rota-Baxter AssDer pair and $(M,\delta_\mathrm{M},T)$ be a representation of it. Then the cocchain complex $(\mathfrak{C}_\mathrm{R.B.AssDer}(\mathfrak{A}_\mathrm{R},M_\mathrm{R}),\mathfrak{d}_{R.A.D})$ is called the cochain complex of weighted Rota-Baxter operator $R$ with coefficients in $(M,\delta_\mathrm{M},T)$, denoted by $\mathfrak{C}_\mathrm{R.AssDer}^\star(A,M)$ and its comology is denoted by $\mathcal{H}_\mathrm{R.B.AssDer}(A,M)$, are called the cohomology of weighted Rota-Baxter operator $R$ with coefficients in $(M,\delta_\mathrm{M},T)$.
\end{defn} 
\begin{re}
	When $(M,\delta_\mathrm{M},T)=(A,\delta,R)$ is the adjoint representation, then we denote the cochain complex of weighted Rota-Baxter operator $R$ by $\mathfrak{C}^\star_\mathrm{R.B.AssDer}(A)$ and its cohomology groups are denoted simply by $\mathcal{H}^\star_\mathrm{R.B.AssDer}(A)$.
\end{re}

At the final step, we will combine the cohomology of AssDer pairs with the cohomology of weighted  Rota-Bxater operators to define a cohomology theory for weighted Rota-Baxter AssDer pair.\\
Define the map $\Phi^\star:C^\star(A,M)\rightarrow C^\star(\mathfrak{A}_\mathrm{R},M_\mathrm{R})$ by \\
$\Phi^0=Id_\mathrm{M}$ and for $n\geq1$ and $f\in C^n(A,M)$ define $\Phi^n(f)\in C^n(\mathfrak{A}_\mathrm{R},M_\mathrm{R})$ as follow 
\begin{equation*}
	\Phi^n(f)(x_1,\cdots,x_n)=f(Rx_1,\cdots,Rx_n)-\displaystyle\sum_{k=0}^{n-1}\lambda^{n-k-1}\displaystyle\sum_{i_1<\ldots<i_k}T\circ f(x_1,\cdots,R(x_{i_1}),\cdots,R(x_{i_k}),\cdots,x_n)
\end{equation*}
From (\cite{K0}, proposition 5.2), the map $\Phi^\star$ is a chain map.
\begin{defn}
	Let $(M,\delta_\mathrm{M},T)$ be a representation of a weighted Rota-Baxter AssDer pair $(\mathfrak{A},\delta,R)$. Define the cochain complex $(\mathfrak{C}^\star_\mathrm{R.B.AssDer}(A,M),\mathfrak{D}_\mathrm{R.B.AssDer})$ of weighted Rota-Baxter AssDer pair $(\mathfrak{A},\delta,R)$ with coeffecients in $(M,\delta_\mathrm{M},T)$ such that 
	\begin{equation*}
		\mathfrak{C}^0_\mathrm{R.B.AssDer}(A,M)=0 \text{ and } \mathfrak{C}^n_\mathrm{R.B.AssDer}(A,M):=\mathfrak{C}^n_\mathrm{AssDer}(A,M)\oplus C^{n-1}(\mathfrak{A}_\mathrm{R},M_\mathrm{R}),\quad \forall n\geq1.
	\end{equation*}
With  $\mathfrak{C}^n_\mathrm{R.B.AssDer}(A,M)=(C^n(A,M)\oplus C^{n-1}(A,M))\oplus C^{n-1}(\mathfrak{A}_\mathrm{R},M_\mathrm{R})$.\\
And the differential  
\begin{equation*}
	\mathfrak{D}_\mathrm{R.B.AssDer}:\mathfrak{C}^n_\mathrm{R.B.AssDer}(A,M)\rightarrow \mathfrak{C}^{n+1}_\mathrm{R.B.AssDer}(A,M) \text{ is given by }
\end{equation*}
\begin{equation}
	\mathfrak{D}_\mathrm{R.B.AssDer}((f,g),h):=(d_\mathrm{Hoch}(f),d_\mathrm{Hoch}(g)+(-1)^n\partial f,-d_\mathrm{R.Hoch}(h)-\Phi^n(f)).
\end{equation}
Where $d_\mathrm{Hoch}:C^n(A,M)\rightarrow C^{n+1}(A,M)$ is the coboundary map associated to the Hochschild cohomology \cite{G2}.
\end{defn}
The cohomology of weighted Rota-Baxter AssDer pair $(\mathfrak{C}^\star_\mathrm{R.B.AssDer}(\mathfrak{A}_\mathrm{R},M_\mathrm{R}),\mathfrak{D}_\mathrm{R.B.AssDer})$ of $(\mathfrak{A},\delta,R)$ with coefficients in $(M,\delta_\mathrm{M},T)$ is denoted by $\mathcal{H}^\star_\mathrm{R.B.AssDer}(A)$
\subsection{Relation between cohomology of weighted Rota-Baxter AssDer and LieDer pairs}
In this section, we show that the cohomology of Rota-Baxter LieDer pair is related to the cohomology of Rota-Baxter AssDer pairs via a suitable skew-symmetrization.
\begin{prop}
	Let $(\mathfrak{A},\delta,R)$ be a weighted Rota-Baxter AssDer pair. Then $(\mathfrak{A}_\mathrm{c},\delta,R)$ is a weighted Rota-Baxter LieDer pair, where $\mathfrak{A}_\mathrm{c}=(A,[\cdot,\cdot]_\mathrm{c})$ such that 
	\begin{equation*}
		[x,y]_\mathrm{c}=\mu(x,y)-\mu(y,x),\quad \forall x,y \in A.
	\end{equation*}
	(is called the skew-symmetrization).
\end{prop}
\begin{re}
	Moreover, if $(M,\delta_\mathrm{M},T)$ is a representation of $(\mathfrak{A},\delta,R)$ then $(M_\mathrm{c},\delta_\mathrm{M},T)$ is a representation of the weighted Rota-Baxter LieDer pair $(\mathfrak{A}_\mathrm{c},\delta,R)$ with the representation is given as follow 
	\begin{equation*}
		\rho(x)(m)=l(x,m)-r(m,x),\quad \forall x\in A, m\in M.
	\end{equation*}
\end{re}
It is known that the standard skew-symmetrization gives rise to a morphism from the Hochschild cochain
complex of an associative algebra to the Chevalley-Eilenberg cochain complex of the corresponding skewsymmetrized Lie algebra. Means that there is a morphism $S_n$ from the $n$-th cochain group (in context of Hochschild cohomology) $C^n(\mathfrak{A},M)$ to the $n$-th cochain group $C^n(\mathfrak{A}_\mathrm{c},M_\mathrm{c})$ (in the context of Chevalley-Eilenberg cohomology), more general $S_\star$ are called the skew-symmetrization maps.
\begin{thm}
	Let $(\mathfrak{A},\delta,R)$ be a weighted Rota-Baxter AssDer pair and $(M,\delta_\mathrm{M},T)$ be a representation of it. Then the collection of maps 
	\begin{equation*}
		\mathcal{S}_\mathrm{n}:\mathfrak{C}^n_\mathrm{R.B.AssDer}(\mathfrak{A},M)\rightarrow \mathfrak{C}^n_\mathrm{R.B.LieDer}(\mathfrak{A}_\mathrm{c},M_\mathrm{c}),\quad \mathcal{S}_\mathrm{n}=(S_n,S_{n-1},S_{n-1})
	\end{equation*}
	induces a morphism from the cohomology of $(\mathfrak{A},\delta,R)$ with coefficients in $(M,\delta_\mathrm{M},T)$ to the cohomology of $(\mathfrak{A}_\mathrm{c},\delta,R)$ with coefficients in $(M_\mathrm{c},\delta_\mathrm{M},T)$.
\end{thm}
\begin{proof}
	For $((f,g),h)\in \mathfrak{C}^n_\mathrm{R.B.AssDer}(\mathfrak{A},M)$ we have 
	\begin{align*}
		\mathfrak{D}_\mathrm{R.B.LieDer}\circ \mathcal{S}_\mathrm{n} ((f,g),h)&=\mathfrak{D}_\mathrm{R.B.LieDer}(S_n(f),S_{n-1}(g),S_{n-1}(h))\\
		&=(d\circ S_n(f),d\circ S_{n-1}(g)+(-1)^n\partial \circ S_n(f),-d_\mathrm{R}\circ S_{n-1}(h)-\Phi\circ S_n(f))\\
		&=(S_{n+1}\circ d_\mathrm{Hoch}(f),S_n\circ d_\mathrm{Hoch}(g)+(-1)^nS_{n+1}\circ \partial (f),-S_n\circ d_\mathrm{R.Hoch}(h)-S_{n+1}\circ \Phi^n(f))\\
		&=\mathcal{S}_{n+1}\circ \mathfrak{D}_\mathrm{R.B.AssDer}((f,g),h)
	\end{align*}
\end{proof}

\section{Aplication of cohomology: Formal deformations}\label{section 4}
In this section, we will study formal deformation of weighted Rota-Baxter LieDer and AssDer pairs.
\subsection{Formal deformation of weighted Rota-Baxter LieDer pairs}
In this subsection we will deform the Lie bracket on $L$, the Rota-Baxter operator $R$ and the derivation $\delta$. We investigate the relation between such deformation and cohomology of $(L,\delta,R)$ with coefficients in the adjoint representation.\\
Let $(\mathfrak{L},\delta,R)$ be a Rota-Baxter LieDer pair, let $\gamma\in C^2(L,L)=\mathrm{Hom}(\wedge^2L,L)$ be the element that correspponds to the Lie bracket on $L$, i.e., $\gamma(x,y)=[x,y]$ for $x,y\in L$. Consider the space $L[[t]]$ of the formal power series in $t$ with coefficients from $L$. Then $L[[t]]$ is a $\mathbb{K}[[t]]$-module.
\begin{defn}
	A formal $1$-parameter deformation of $(\mathfrak{L},\delta,R)$ consists of a triple $(\gamma_\mathrm{t},\delta_\mathrm{t},R_\mathrm{t})$ of three formal power series 
	\begin{eqnarray*}
		\gamma_\mathrm{t}&=&\displaystyle\sum_{i\geq0}\gamma_\mathrm{i}t^\mathrm{i},\quad \text{where } \gamma_\mathrm{i}\in \mathrm{Hom}(\wedge^2L,L) \text{ with } \gamma_0=\gamma,\\
		\delta_\mathrm{t}&=&\displaystyle\sum_{i\geq0}\delta_\mathrm{i}t^\mathrm{i},\quad \text{where } \delta_\mathrm{i}\in \mathrm{Hom}(L,L) \text{ with } \delta_0=\delta,\\
		R_\mathrm{t}&=&\displaystyle\sum_{i\geq0}R_\mathrm{i}t^\mathrm{i},\quad \text{where } R_\mathrm{i}\in \mathrm{Hom}(L,L) \text{ with } R_0=R.
	\end{eqnarray*}
Such that the $\mathbb{K}[[t]]$-bilinear map $\gamma_\mathrm{t}$ defines a Lie algebra structure on $L[[t]]$ and the two $\mathbb{K}[[t]]$-linear maps  $\delta_\mathrm{t},R_\mathrm{t}:L[[t]] \rightarrow L[[t]]$ are respectively a derivation and a Rota-Baxter opperator. This means that $(L[[t]]=(L[[t]],\gamma_\mathrm{t}),\delta_\mathrm{t,R_\mathrm{t}})$ is a Rota-Baxter LieDer pair over $\mathbb{K}[[t]]$. 
\end{defn}
Thus, $(\gamma_\mathrm{t},\delta_{\mathrm{t}},R_\mathrm{t})$ is a formal $1$-parameter deformation of $(\mathfrak{L},\delta,R)$ if and only if, for $x,y,z\in L$
\begin{eqnarray*}
	\circlearrowleft_{x,y,z}\gamma_\mathrm{t}(x,\gamma_\mathrm{t}(y,z))&=&0,\\
	\gamma_\mathrm{t}(R_\mathrm{t}x,R_\mathrm{t}y)&=&R_\mathrm{t}(\gamma_\mathrm{t}(R_\mathrm{t}x,y)+\gamma_\mathrm{t}(x,R_{\mathrm{t}}y)+\lambda\gamma_\mathrm{t}(x,y) ),\\
	\delta_{\mathrm{t}}(\gamma_\mathrm{t}(x,y))&=&\gamma_\mathrm{t}(\delta_\mathrm{t}x,y)+\gamma_\mathrm{t}(x,\delta_{\mathrm{t}}y). 
\end{eqnarray*}
They are equivalent to the following systems of identities, for $n\geq0$ and for all $x,y,z\in L$,
\begin{eqnarray}
	\displaystyle\sum_{i+j=n}\gamma_\mathrm{i}(x,\gamma_\mathrm{j}(y,z))+\gamma_\mathrm{i}(y,\gamma_\mathrm{j}(z,x))+\gamma_\mathrm{i}(z,\gamma_\mathrm{j}(x,y))&=&0,\label{eqt deformation 1}\\
	\displaystyle\sum_{i+j=n}\Big(\delta_{\mathrm{i}}(\gamma_\mathrm{j}(x,y))-\gamma_\mathrm{j}(\delta_{\mathrm{i}}x,y)-\gamma_\mathrm{j}(x,\delta_{\mathrm{i}}y) \Big)&=&0,\label{eqt deformation 2}\\
	\displaystyle\sum_{i+j+k=n}\Big(\gamma_\mathrm{i}(R_\mathrm{j}x,R_\mathrm{k}y) -R_\mathrm{i}(\gamma_\mathrm{j}(R_\mathrm{k}x,y)+\gamma_\mathrm{j}(x,R_\mathrm{k}y)+\lambda\gamma_\mathrm{j}(x,y))\Big)&=&0.\label{eqt deformation 3}
\end{eqnarray}
\begin{defn}
	Two formal deformations $(\gamma_\mathrm{t},\delta_{\mathrm{t}},R_\mathrm{t})$ and $(\gamma^{\prime}_\mathrm{t},\delta^\prime_{\mathrm{t}},R^\prime_\mathrm{t})$ of a Rota-Baxter LieDer pair $(L,\delta,R)$ are said to be  equivalent if there is a formal isomorphism 
	\begin{equation*}
		\varphi_\mathrm{t}=\displaystyle\sum_{i\geq0}\varphi_\mathrm{i}t^\mathrm{i}:L[[t]]\rightarrow L[[t]],\quad \text{where } \varphi_\mathrm{i}\in \mathrm{Hom}(L,L) \text{ and } \varphi_0=\mathrm{id}_L.
	\end{equation*}
Such that the $\mathbb{K}[[t]]$-linear map $\varphi_\mathrm{t}$ is a morphism of Rota-Baxter LieDer pairs from $(L[[t]]^\prime,\delta^\prime_{\mathrm{t}},R^\prime_\mathrm{t})$ to $(L[[t]],\delta_{\mathrm{t}},R_\mathrm{t})$
\end{defn}
It means that the following identities holds 
\begin{equation*}
	\varphi_\mathrm{t}(\gamma^{\prime}_\mathrm{t}(x,y))=\gamma_\mathrm{t}(\varphi_\mathrm{t}(x),\varphi_\mathrm{t}(y)) \text{ and } \varphi_\mathrm{t}\circ \delta^\prime_{\mathrm{t}}=\delta_{\mathrm{t}} \circ \varphi_\mathrm{t} \text{ and } \varphi_\mathrm{t}\circ R^\prime_\mathrm{t}=R_\mathrm{t}\circ \varphi_\mathrm{t}.
\end{equation*}
They can be expressed as follow, for $n\geq0$
\begin{eqnarray}
	\displaystyle\sum_{i+j=n}\varphi_\mathrm{i}(\gamma^\prime_\mathrm{j}(x,y))&=&\displaystyle\sum_{i+j+k=n}\gamma_\mathrm{i}(\varphi_\mathrm{j}(x),\varphi_\mathrm{k}(y)),\label{eqt deformation 4}\\
	\displaystyle\sum_{i+j=n}\varphi_\mathrm{i}\circ \delta^\prime_{\mathrm{j}}&=&\displaystyle\sum_{i+j=n}\delta_\mathrm{i}\circ \varphi_\mathrm{j},\label{eqt deformation 5}\\
		\displaystyle\sum_{i+j=n}\varphi_\mathrm{i}\circ R^\prime_{\mathrm{j}}&=&\displaystyle\sum_{i+j=n}R_\mathrm{i}\circ \varphi_\mathrm{j}.\label{eqt deformation 6}
\end{eqnarray}
\begin{thm}\label{Theorem deformation}
	Let $(\gamma_\mathrm{t},\delta_{\mathrm{t}},R_\mathrm{t})$ be a formal $1$-parameter deformation of the Rota-Baxter LieDer pair $(L,\delta,R)$. Then $(\gamma_1,\delta_1,R_1)\in \mathfrak{C}^2_\mathrm{RBLieDer}(L,L)$ is a $2$-cocycle in the cohmology of $(L,\delta,R)$ with coefficients in the adjoint representation.\\
\end{thm}
\begin{proof}
	Since $(\gamma_\mathrm{t},\delta_{\mathrm{t}},R_\mathrm{t})$ is a formal $1$-parameter deformation, we have from \eqref{eqt deformation 1}, \eqref{eqt deformation 2} and \eqref{eqt deformation 3}, for $n\geq0$, that 
	\begin{equation*}
		\circlearrowleft_{x,y,z}[x,\gamma_1(y,z)]+\circlearrowleft_{x,y,z}\gamma_1(x,[y,z])=0,
	\end{equation*}
and 
\begin{equation*}
	\gamma_1(Rx,Ry)+[R_1x,Ry]+[Rx,R_1y]=R_1([Rx,y]+[x,Ry])+R([R_1x,y]+[x,R_1y]+\gamma_1(Rx,y)+\gamma_1(x,Ry)),
\end{equation*}
and
\begin{equation*}
	\delta_1([x,y])+\delta(\gamma_1(x,y))-\gamma_1(\delta x,y)-[\delta_1 x,y]-\gamma_1[x,\delta y]-[x,\delta_1y]=0.
\end{equation*}
The first identity means that $\mathrm{d}\gamma_1=0$ and the second identity means that $\mathrm{d}_\mathrm{R}(R_1)+\Phi^2(\gamma_1)=0$ and the last one means that $\mathrm{d}(\delta_1)+\partial^2(\gamma_1)=0$
This implies that 
\begin{align*}
	\mathfrak{D}_\mathrm{RBLieDer}((\gamma_1,\delta_1),R_1)&=(\mathrm{d}(\gamma_1),\mathrm{d}(\delta_1)+\partial(\gamma_1),-\partial(R_1)-\Phi(R_1))\\
	&=(0,0,-\partial(R_1)-\Phi(R_1))\\
	&=(0,0,0).
\end{align*}
This complete the proof.
\end{proof}
\begin{thm}
	Let $(L,\delta,R)$ be a weighted Rota-Baxter LieDer pair. If $\mathcal{H}_\mathrm{R.B.LieDer}(L,L)=0$ then any formal $1$-parameter deformation of $(\mathfrak{L},\delta,R)$ is equivalent to the trivial one $(\gamma^\prime_{\mathrm{t}}=\gamma,R^\prime_{\mathrm{t}}=R,\delta^\prime_{\mathrm{t}}=\delta)$.
\end{thm}
\begin{proof}
Let $(\gamma_\mathrm{t},\delta_\mathrm{t},R_\mathrm{t})$ be a formal $1$-parameter deformation of the weighted Rota-Baxter LieDer pair $(\mathfrak{L},\delta,R)$. From theorem \ref{Theorem deformation} we have that $(\gamma1,\delta1,R_1)$ is a $2$-cocycle. Then from the hypothesis there exists $f\in \mathrm{Hom}(L,L)$ such that 
\begin{equation}\label{equation deformation}
	\mathfrak{D}_\mathrm{R.B.LieDer}(f)=(d(f),-\partial(f),-\Phi(f)).
\end{equation}
Means that $(\gamma1,\delta1,R_1)=(d(f),-\partial(f),-\Phi(f))$.\\
Let $\phi_\mathrm{t}:L\rightarrow L$ be the map $\phi_\mathrm{t}=id+\phi_1 t$. Then $(\overline{\gamma_\mathrm{t}}=\phi_\mathrm{t}^{-1}\circ \gamma_\mathrm{t}\circ(\phi_\mathrm{t}\otimes\phi_\mathrm{t}),\overline{\delta_\mathrm{t}}=\phi_\mathrm{t}^{-1}\circ\delta_\mathrm{t}\circ\phi_\mathrm{t},\overline{R_\mathrm{t}}=\phi_\mathrm{t}^{-1}\circ R_\mathrm{t}\circ\phi_\mathrm{t} )$ is a formal $1$-deformation of $(\mathfrak{L},\delta,R)$ equivalent to $(\gamma_\mathrm{t},\delta_\mathrm{t},R_\mathrm{t})$. By \eqref{equation deformation} we can easily check that $(\overline{\gamma_1}=\overline{\delta_1}=\overline{R_1}=0)$, means that 
\begin{eqnarray*}
	\overline{\gamma_\mathrm{t}}&=&\gamma+\gamma_2t^2+...\\
	\overline{\delta_\mathrm{t}}&=&\delta+\delta_2 t^2+...\\
	\overline{R_\mathrm{t}}&=&R+R_2t^2+...
\end{eqnarray*}
and by repeating the same argument we conculde that $(\gamma_\mathrm{t},\delta_\mathrm{t},R_\mathrm{t})$ is equivalent to $(\gamma^\prime_{\mathrm{t}}=\gamma,R^\prime_{\mathrm{t}}=R,\delta^\prime_{\mathrm{t}}=\delta)$.
\end{proof}
\subsection{Formal deformation of weighted Rota-Baxter AssDer pairs}
In this section we study formal deformation of weighted Rota-Baxter AssDer pairs.\\
Let $(\mathfrak{A},\delta,R)$ be a Rota-Baxter AssDer pair. A formal $1$-parameter deformation of $(\mathfrak{A},\delta,R)$ consists of three formal power series 
	\begin{eqnarray*}
	\mu_\mathrm{t}&=&\displaystyle\sum_{i\geq0}\mu_\mathrm{i}t^\mathrm{i},\quad \text{where } \mu_\mathrm{i}\in \mathrm{Hom}(\wedge^2L,L) \text{ with } \mu_0=\mu,\\
	\delta_\mathrm{t}&=&\displaystyle\sum_{i\geq0}\delta_\mathrm{i}t^\mathrm{i},\quad \text{where } \delta_\mathrm{i}\in \mathrm{Hom}(L,L) \text{ with } \delta_0=\delta,\\
	R_\mathrm{t}&=&\displaystyle\sum_{i\geq0}R_\mathrm{i}t^\mathrm{i},\quad \text{where } R_\mathrm{i}\in \mathrm{Hom}(L,L) \text{ with } R_0=R.
\end{eqnarray*}
Then we say that $(\mu_\mathrm{t},\delta_\mathrm{t},R_\mathrm{t})$ is a formal $1$-parameter deformation of $(\mathfrak{A},\delta,R)$ if and only if 
\begin{eqnarray*}
    \mu_\mathrm{t}(\mu_\mathrm{t}(x,y),z)&=&\mu_\mathrm{t}(x,\mu_\mathrm{t}(y,z)),\\
	\mu_\mathrm{t}(R_\mathrm{t}x,R_\mathrm{t}y)&=&R_\mathrm{t}(\mu_\mathrm{t}(R_\mathrm{t}x,y)+\mu_\mathrm{t}(x,R_{\mathrm{t}}y)+\lambda\mu_\mathrm{t}(x,y) ),\\
	\delta_{\mathrm{t}}(\mu_\mathrm{t}(x,y))&=&\mu_\mathrm{t}(\delta_\mathrm{t}x,y)+\mu_\mathrm{t}(x,\delta_{\mathrm{t}}y). 
\end{eqnarray*}
And they are equivalent to the followings
\begin{eqnarray}
	\displaystyle\sum_{i+j=n}\mu_\mathrm{i}(\mu_\mathrm{j}(x,y),z)-\mu_\mathrm{i}(x,\mu_\mathrm{j}(y,z))&=&0,\label{eqt deformation Ass1}\\
	\displaystyle\sum_{i+j=n}\Big(\delta_{\mathrm{i}}(\mu_\mathrm{j}(x,y))-\mu_\mathrm{j}(\delta_{\mathrm{i}}x,y)-\mu_\mathrm{j}(x,\delta_{\mathrm{i}}y) \Big)&=&0,\label{eqt deformation Ass2}\\
	\displaystyle\sum_{i+j+k=n}\Big(\mu_\mathrm{i}(R_\mathrm{j}x,R_\mathrm{k}y) -R_\mathrm{i}(\mu_\mathrm{j}(R_\mathrm{k}x,y)+\mu_\mathrm{j}(x,R_\mathrm{k}y)+\lambda\mu_\mathrm{j}(x,y))\Big)&=&0.\label{eqt deformation Ass3}
\end{eqnarray}
For $n=1$ we obtain 
\begin{equation*}
	\mu_1(\mu(x,y),z)+\mu(\mu_1(x,y),z)=\mu(x,\mu_1(y,z))+\mu_1(x,\mu(y,z))
\end{equation*}
which means that $d_\mathrm{Hoch}(\mu_1)=0$.
And the second equation 
\begin{equation*}
	\delta(\mu_1(x,y))+\delta_1(\mu(x,y))=\mu(\delta_1x,y)+\mu_1(\delta x,y)+\mu(x,\delta_1y)+\mu_1(x,\delta y)
\end{equation*}
means that $d_\mathrm{Hoch}(\delta_1)+\partial \mu_1=0$.  And the third equation
\begin{eqnarray*}
	&&\mu_1(Rx,Ry)-R_1(\mu(Rx,y)+\mu(x,Ry)+\lambda \mu(x,y))\\
	&&+\mu(R_1x,Ry)-R(\mu_1(Rx,y)+\mu_1(x,Ry)+\lambda \mu_1(x,y))\\
	&&+\mu(Rx,Ry)-R(\mu(R_1x,y)+\mu(x,R_1y)+\lambda \mu(x,y))=0
\end{eqnarray*}
which means that  $-d_\mathrm{R.Hoch}-\Phi^2(\mu_1)=0$. which leads us to the following
\begin{prop}
	Let $(\mu_\mathrm{t},\delta_\mathrm{t},R_\mathrm{t})$ be a formal deformation of a weighted Rota-Baxter AssDer pair $(\mathfrak{A},\delta,R)$. Then the linear term $(\mu_1,\delta_1,R_1)$ is a $2$-cocycle in the cohomology of the weighted Rota-Baxter AssDer pair $(\mathfrak{A},\delta,R)$ with coefficients in itself.
\end{prop}
\begin{defn}
	Two formal deformations $(\mu_\mathrm{t},\delta_{\mathrm{t}},R_\mathrm{t})$ and $(\mu^{\prime}_\mathrm{t},\delta^\prime_{\mathrm{t}},R^\prime_\mathrm{t})$ of a weighted Rota-Baxter AssDer pair $(\mathfrak{A},\delta,R)$ are said to be  equivalent if there is a formal isomorphism 
	\begin{equation*}
		\varphi_\mathrm{t}=\displaystyle\sum_{i\geq0}\varphi_\mathrm{i}t^\mathrm{i}:A[[t]]\rightarrow A[[t]],\quad \text{where } \varphi_\mathrm{i}\in \mathrm{Hom}(L,L) \text{ and } \varphi_0=\mathrm{id}_L.
	\end{equation*}
	Such that the $\mathbb{K}[[t]]$-linear map $\varphi_\mathrm{t}$ is a morphism of Rota-Baxter LieDer pairs from $(A[[t]]^\prime,\delta^\prime_{\mathrm{t}},R^\prime_\mathrm{t})$ to $(A[[t]],\delta_{\mathrm{t}},R_\mathrm{t})$
\end{defn}
It means that the following identities holds 
\begin{equation*}
	\varphi_\mathrm{t}(\mu^{\prime}_\mathrm{t}(x,y))=\mu_\mathrm{t}(\varphi_\mathrm{t}(x),\varphi_\mathrm{t}(y)) \text{ and } \varphi_\mathrm{t}\circ \delta^\prime_{\mathrm{t}}=\delta_{\mathrm{t}} \circ \varphi_\mathrm{t} \text{ and } \varphi_\mathrm{t}\circ R^\prime_\mathrm{t}=R_\mathrm{t}\circ \varphi_\mathrm{t}.
\end{equation*}
It is equivalent to the followings 
\begin{eqnarray*}
	\displaystyle\sum_{i+j=n}\varphi_\mathrm{i}(\mu^\prime_\mathrm{j}(x,y))&=&\displaystyle\sum_{i+j+k=n}\mu_\mathrm{i}(\varphi_\mathrm{j}(x),\varphi_\mathrm{k}(y)),\\
	\displaystyle\sum_{i+j=n}\varphi_\mathrm{i}\circ \delta^\prime_{\mathrm{j}}&=&\displaystyle\sum_{i+j=n}\delta_\mathrm{i}\circ \varphi_\mathrm{j},\\
	\displaystyle\sum_{i+j=n}\varphi_\mathrm{i}\circ R^\prime_{\mathrm{j}}&=&\displaystyle\sum_{i+j=n}R_\mathrm{i}\circ \varphi_\mathrm{j}.
\end{eqnarray*}
For $n=0$ we have $\varphi_0=\mathrm{Id_\mathrm{A}}$ and for $n=1$ we obtain 
\begin{eqnarray}
	\varphi_1\circ \mu^\prime+\mu_1^\prime&=&\mu_1+\mu\circ (\varphi_1\otimes \mathrm{Id_A})+\mu\circ (\mathrm{Id_A}\otimes \varphi_1),\label{eqt deformation Ass4}\\
	\varphi_1\circ\delta^\prime+\delta^\prime_1&=&\delta_1+\delta\circ\varphi_1,\label{eqt deformation Ass5}\\
	\varphi_1\circ R^\prime+R^\prime_1&=&R_1+R\circ\varphi_1.\label{eqt deformation Ass6}
\end{eqnarray}
Then equations \eqref{eqt deformation Ass4},\eqref{eqt deformation Ass5} and \eqref{eqt deformation Ass6} we obtain that 
\begin{equation*}
     (\mu^\prime_1,\delta^\prime_1,R^\prime_1)-(\mu_1,\delta_1,R_1):=\mathfrak{D}_\mathrm{R.B.AssDer}(\varphi_1)
\end{equation*}
This leads us to the following result
\begin{thm}
	Tow formal $1$-parameter deformations of a weighted Rota-Baxter assDer pair $(\mathfrak{A},\delta,R)$ are cohomologous. Therefore, they correspond to the same cohomology class.
\end{thm}
\begin{defn}
	A formal deformation $(\mu_\mathrm{t},\delta_\mathrm{t},R_\mathrm{t})$ of a weighted Rota-Baxter AssDer pair $(\mathfrak{A},\delta,R)$ is said trivial if it is equivalent to $(\mu^\prime_\mathrm{t}=\mu;\delta^\prime_\mathrm{t}=\delta)$.
\end{defn}
\begin{thm}
	If $\mathcal{H}^2_\mathrm{R.B.AssDer}(A,A)=0$ then every formal deformation of the weighted Rota-Baxter AssDer pair $(\mathfrak{A},\delta,R)$ is trivial.
\end{thm}
 	\noindent {\bf Acknowledgment:}
	The authors would like to thank the referee for valuable comments and suggestions on this article. 

\end{document}